\documentclass[a4paper,12pt]{amsart}

\usepackage{amsmath,epsfig,amssymb,amsbsy,amscd,amsfonts,amstext,color,bm}

\usepackage[all]{xy}
\usepackage{bm}
\usepackage{comment}
\usepackage{enumerate}
\usepackage{circledsteps}
\usepackage{geometry}

\makeatletter
\@namedef{subjclassname@2020}{%
  \textup{2020} Mathematics Subject Classification}
\makeatother

\usepackage{tikz,epic,eso-pic}

\tikzset{v/.style={circle, draw, inner sep=2pt, minimum size=6pt, fill=white}}
\usetikzlibrary{patterns}

\theoremstyle{plain}
\newtheorem{theorem}{Theorem}[section]
\newtheorem*{theorem*}{Theorem}
\newtheorem*{theoremA*}{Theorem A}
\newtheorem*{theoremB*}{Theorem B}
\newtheorem{corollary}[theorem]{Corollary}
\newtheorem{lemma}[theorem]{Lemma}

\theoremstyle{definition}
\newtheorem{definition}[theorem]{Definition}
\newtheorem{remark}[theorem]{Remark}

\newtheorem{proposition}[theorem]{Proposition}

\newcommand{\Int}{\operatorname{Int}}

\newcommand{\Hom}{\operatorname{Hom}}

\newcommand{\Sing}{\operatorname{Sing}}

\newcommand{\rank}{\operatorname{rank}}

\newcommand{\codim}{\operatorname{codim}}

\newcommand{\bd}{\partial}

\newcommand{\bC}{\mathbb{C}}

\newcommand{\bR}{\mathbb{R}}

\newcommand{\bZ}{\mathbb{Z}}

\newcommand{\cA}{\mathcal{A}}
\newcommand{\cB}{\mathcal{B}}
\newcommand{\cC}{\mathcal{C}}

\def\qed{\hfill $\Box$}

\title{Topological line arrangements and their topological invariants}
\author{Sakumi Sugawara}
\date{\today}
\address{Department of Mathematics, Faculty of Science, Hokkaido University, North 10, West 8, Kita-ku, Sapporo 060-0810, JAPAN. }
\email{sugawaras@math.sci.hokudai.ac.jp}

\begin{document}
\begin{abstract}
A topological line arrangement is an arrangement of embedded spheres in the complex projective plane that topologically generalizes a complex line arrangement. 
In this paper, we establish foundational results on the topology of the complement of topological line arrangements.
First, we prove that the cohomology ring of the complement is isomorphic to the Orlik-Solomon algebra, as for classical complex line arrangements. 
We then study the homotopy type of the complement.
We prove that the complement of a symplectic line arrangement has the homotopy type of a minimal CW complex. In contrast, every combinatorial type realizable by a topological line arrangement admits a realization with a non-minimal complement. Moreover, every such combinatorial type admits infinitely many realizations whose complements are pairwise non-homotopy equivalent.
\end{abstract}

\maketitle

%\renewcommand{\thefootnote}{\fnsymbol{footnote}}
%\footnote[0]{MSC Classification: 32S50, 32Q55, 57K40, 57R65}
%32S50: Topological aspects of complex singularities: Lefschetz theorems, topological classification, invariants
%32Q55: Topological aspects of complex manifolds
%57K40: General topology of 4-manifolds
%57R65: Surgeries and handlebodies
%\footnote[0]{Keywords:}
%\renewcommand{\thefootnote}{\arabic{footnote}}

\section{Introduction}
A \textit{hyperplane arrangement} is a finite collection of hyperplanes in a linear, affine, or projective space. Hyperplane arrangements have been studied from many different viewpoints, including algebraic geometry, combinatorics, and topology \cite{orl-ter}. 
From a topological viewpoint, one of the central problems is to understand to what extent the combinatorics of the arrangement determines the topology of the complement. 
One of the most fundamental results in this direction is the combinatorial description of the cohomology ring of the complement.
This line of work began with Arnol'd's computation for the braid arrangement complement \cite{arn}.
Brieskorn then showed that, for an arbitrary complex hyperplane arrangement, the cohomology ring is generated by logarithmic differential forms associated with defining linear equations \cite{bri}. 
Finally, Orlik and Solomon gave a purely combinatorial presentation of the cohomology rings in terms of the intersection lattice of the arrangement \cite{orl-sol}. 
Since then, Orlik-Solomon type descriptions have been extended to several broader classes of arrangements, 
including subspace arrangements \cite{gor-mac, fei-zie, del-sch}, toric arrangements \cite{cal-del, cddmp}, and abelian arrangements \cite{bib, moc-pag, bpp}. 

Moreover, complex hyperplane arrangement complements have a distinguished topological property: \textit{minimality}. 
A finite CW complex is called \textit{minimal} if the number of $k$-cells equals the $k$-th Betti number for each $k \geq 0$.
Dimca-Papadima and Randell proved independently that the complement of a complex hyperplane arrangement has the homotopy type of a minimal CW complex \cite{dim-pap, ran-mor}.
This is a particularly notable feature of arrangement complements and does not hold for general complex hypersurface complements.
Minimality has also been extended beyond hyperplane arrangements, for instance to certain subspace arrangements \cite{adi-min}  
and toric arrangements \cite{ant-del}.

The purpose of this paper is to investigate the extent to which these classical topological properties of complex line arrangements extend to a more flexible object: \textit{topological line arrangements}.

A topological line arrangement is a finite collection of locally flat embedded $2$-spheres in $\bC P^2$ such that any two spheres intersect transversally and positively in exactly one point. 
Thus, it is a topological analogue of a complex projective line arrangement.
The study of topological line arrangements was initiated by Ruberman and Starkston \cite{rub-sta}, who established various fundamental properties. 
It may also be viewed as a complex analogue of a pseudoline arrangement in the real projective plane. 
Indeed, Ruberman and Starkston proved that pseudoline arrangements can be complexified, in an appropriate sense, to produce topological line arrangements \cite[Theorem 1.3]{rub-sta}.
Thus the class of topological line arrangements is strictly larger than the class of complex line arrangements. 
However, not every rank-three matroid can be realized as a topological line arrangement \cite[Theorems 1.1 and 1.2]{rub-sta}. 
Topological line arrangements also arise in symplectic geometry, for instance in the work of Plamenevskaya and Starkston \cite{pla-sta}, and four-dimensional topological methods have recently been applied to their realization problems \cite{gol,ace-gol, AMP}.
There is also a related notion of \textit{$2$-pseudoarrangements} introduced by Bj\"orner and Ziegler \cite{bjo-zie}; in this paper we will compare this notion with topological line arrangements and show that, in a natural sense, topological line arrangements form a genuinely broader class (see Theorem \ref{thm:lifted}).

Our goal is to determine which classical topological properties of complex line arrangement complements remain valid for complements of topological line arrangements. Our first main result concerns the cohomology ring.

\begin{theorem}\label{thm:main_cohom}
Let $\cA$ be a topological line arrangement and let $U(\cA)$ be the complement. Then, the cohomology ring $H^{*}(U(\cA); \bZ)$ is isomorphic to the (projective) Orlik-Solomon algebra $A^{*} (\cA)$.
\end{theorem}

Thus, at the level of cohomology rings, topological line arrangements satisfy the same combinatorial description as complex line arrangements. In particular, the cohomology groups of $U(\cA)$ are torsion-free, and their Betti numbers are determined by the characteristic polynomial of the arrangement.
We give two topological proofs of Theorem \ref{thm:main_cohom}.
In the first proof, we use Poincar\'e-Lefschetz duality to translate the cup product into a computation of intersection products in relative homology.
We construct explicit homology cycles which play the role dual to the logarithmic forms in the classical theory and compute their intersection products. We then give a shorter, direct computation of the cup product. Using the nbc tori constructed from the local links, the cup product can be computed from the meridians, longitudes, and their linking numbers. 
In particular, the positivity assumption forces the relevant local linking numbers to be $+1$.
These computations give the desired isomorphism with the Orlik--Solomon algebra. The cycles in the first proof are inspired by those introduced in \cite{sug-bd} in the computation of the cohomology ring of boundary manifolds of combinatorial line arrangements.
Although a related philosophy appears in the work of de Longueville and Schultz on complex subspace arrangements \cite{del-sch}, the construction and the intersection computations in the present setting are different.

Our next results concern the homotopy type of the complement, especially minimality. We first prove a positive result for symplectic line arrangements.

\begin{theorem}\label{thm:main_symp}
The complement of a symplectic line arrangement is minimal.
\end{theorem}

Symplectic line arrangements include complexifications of pseudoline arrangements, and hence they realize combinatorial types which cannot be realized by complex line arrangements, such as arrangements violating Pappus' theorem.
The proof uses braid monodromy, which is available not only for algebraic curves but also for symplectic curves \cite{moi-bra, coh-suc-braid, kha-kul}.
It is known that the $2$-dimensional CW complex associated to the presentation of the fundamental group obtained by the braid monodromy gives the homotopy type of the complement \cite{lib-hom}. 
Though Libgober's result describes the complement for only algebraic curves, we adapt his argument to the symplectic setting.
Applying this construction to symplectic line arrangements yields a minimal CW model.

In contrast, the situation changes drastically for arbitrary topological line arrangements.

\begin{theorem}\label{thm:main_nonmin}
Let $\cA$ be any topological line arrangement with at least two lines. Then, there exists a topological line arrangement $\cA'$ with the same combinatorial type as $\cA$ such that the complement $U(\cA')$ is not minimal. 
\end{theorem}

%This shows that, unlike complex and symplectic line arrangement complements, complements of general topological line arrangements can have highly flexible homotopy types.

The construction is based on modifying a given topological line arrangement by taking a connected sum of one line with a knotted $2$-sphere ($2$-knot) in $S^4$. 
By choosing a $2$-knot whose knot group has sufficiently large rank, we prove that the resulting complement is not minimal. Such knots are provided in Lemma \ref{lem:2-knot}.
Moreover, by the same ``knot-connecting-sum construction'', we can prove the following stronger statement.

\begin{theorem}\label{thm:main_infty}
Let $\cA$ be any topological line arrangement with at least two lines. Then, there exist infinitely many topological line arrangements of the same combinatorial type as $\cA$, but the complements are pairwise non-homotopy equivalent.
\end{theorem} 

This result further illustrates the flexibility of topological line arrangement complements.
The same idea enables us to show that any $2$-knot complement can appear as the complement of a topological line arrangement (Theorem \ref{thm:main_knot}).
%This result further illustrates the flexibility of topological line arrangement complements. By contrast, among complex line arrangement complements, the only $2$-knot group that can occur as the fundamental group is $\bZ$, the group of the trivial $2$-knot due to the minimality (see Proposition \ref{prop:min_fund}).
%Thus, topological line arrangements provide many new homotopy types that are invisible in the complex setting.

The paper is organized as follows. 
In Section \ref{sec:prelim}, we review basic notions concerning topological line arrangements, their combinatorics, Poincar\'e polynomials, and the Orlik--Solomon algebra. 
In Section \ref{sec:hom}, we compute the homology groups of complements of topological line arrangements. We first analyze the union of the lines, its tubular neighborhood, and the associated boundary manifold. We then compute the homology groups. We also construct explicit bases for these homology groups. 
Section \ref{sec:cohom} is devoted to the cohomology ring. We give two proofs of Theorem \ref{thm:main_cohom}. First, using Poincar\'e-Lefschetz duality, we translate the cup product into the intersection product in relative homology. We construct relative cycles dual to the homology basis obtained in Section \ref{sec:hom} and calculate their intersections.
We then give a shorter direct computation of the cup product by evaluating it on the nbc tori and using the linking numbers of the local links.
In Section \ref{sec:min}, we study the homotopy type of the complement. We begin by recalling several basic properties of minimal CW complexes. We then prove the minimality of symplectic line arrangement complements, establishing Theorem \ref{thm:main_symp}. After that, we construct topological line arrangements with non-minimal complements and non-homotopy equivalent families, proving Theorems \ref{thm:main_nonmin} and \ref{thm:main_infty}. We also compare topological line arrangements with the $2$-pseudoarrangements of Bj\"orner and Ziegler. 
Section \ref{sec:app} recalls the basic facts about homology intersection rings.

\vspace{3mm}

\textbf{Use for AI.} 
The author used ChatGPT 5.6 Sol for exploratory discussions related to the proof of Theorem \ref{thm:main_nonmin} and language editing. The author subsequently verified all mathematical arguments. 

\vspace{3mm}

\textbf{Acknowledgements.} 
The author would like to thank Jumpei Yasuda for helpful discussions.
This work is supported by JSPS KAKENHI 25K23326.

\vspace{3mm}

\textbf{Convention.}
In this paper, we use a standard orientation convention for transverse intersections and manifolds with boundary, see \cite{gul-pol}. In particular, the boundary of oriented manifolds is oriented by the outward normal first convention.

\section{Preliminaries} \label{sec:prelim}

\subsection{Topological line arrangements}

\begin{definition}
A \textit{topological line} $H$ is a locally flat embedded $2$-sphere in $\bC P^2$ satisfying $[H] = [\bC P^1] \in H_2 (\bC P^2; \bZ)$.
A finite set $\cA = \{H_0, \dots, H_n\}$ is called a \textit{topological line arrangement} if distinct $H_{i}$ and $H_{j}$ intersect in exactly one point transversally and positively for each $i \neq  j$.
We write $U(\cA) = \bC P^2 \setminus \bigcup_{i=0}^{n} H_i$ for its complement.
\end{definition}

Similarly, we can define a smooth/symplectic/algebraic line arrangement.
An algebraic line arrangement is the same notion as a classical complex line arrangement.

\begin{definition}
Let $\cA$ be a topological line arrangement. We write $\Sing (\cA)$ for the set of intersection points of lines in $\cA$. 
We define its subset $\Sing_{i} (\cA) = \{P \mid P \in H_{i}\}$. 
Let $n_{i} = |\Sing_{i} (\cA)|$.
For each $P \in \Sing (\cA)$, we define the subarrangement $\cA_{P} = \{H \in \cA \mid H \ni P\}$.
Let $m_{P} = |\cA_{P}|$.
\end{definition}

%\begin{remark}
%Local model may not be a Hopf link of $m_P$ components
%\end{remark}

\begin{definition}
Let $\cA$ be a topological line arrangement. The \textit{intersection poset} (or the \textit{combinatorial type}) is defined by 
\[
L(\cA) = \left \{ \bigcap_{H \in \cB} H \Biggm | \cB \subset \cA \right \}.
\]
The order is defined as $X \leq Y$ if $X \supset Y$.
\end{definition}
Let $L_{k} (\cA)$ be the set of codimension $k$ intersections. That is, $L_1 (\cA) = \cA$, $L_2 (\cA) = \Sing(\cA)$, and $L_3 (\cA) = \{\emptyset\}$, if $\cA$ is not a pencil. By convention, we set $\codim \emptyset = 3$.

\begin{remark}
In the usual definition of the intersection poset for an affine arrangement, we assume $\emptyset \notin L(\cA)$ \cite[Definition 2.1]{orl-ter}. However, we suppose $\emptyset \in L(\cA)$ in this paper. This is because we would like to consider $L(\cA)$ as the combinatorial type of central arrangements. 
\end{remark}

\begin{definition}
We define the \textit{M\"obius function} $\mu: L(\cA) \to \bZ$ by 
\begin{eqnarray*}
\mu (X) = 
\left \{
\begin{array}{ll}
1 & (X=\bC P^2), \\
-\sum_{Y < X} \mu (Y) & (X > \bC P^2).  
\end{array}
\right .
\end{eqnarray*}
\end{definition}

One has $\mu(H) = -1$, $\mu(P) = m_{P}-1$ for each $H \in \cA$ and $P \in \Sing (\cA)$. Moreover, $\sum_{X \in L(\cA)} \mu (X) = 0$.

\begin{definition}
The \textit{Poincar\'e polynomial} of $\cA$ is defined by
\[
\pi (\cA, t) = \sum_{X \in L(\cA)} \mu (X) (-t)^{\codim X}.
\]
Since $\pi (\cA, -1) = \sum_{X \in L(\cA)} \mu (X) =0$, $(t+1)$ divides $\pi (\cA,t)$. The quotient $\overline{\pi} (\cA,t) = \pi (\cA,t)/ (t+1)$ is called the \textit{reduced Poincar\'e polynomial}.
\end{definition}

We have the following explicit expansion of the reduced Poincar\'e polynomial.

\begin{proposition}\label{prop:exp_poin}
Let $\cA =\{H_0,\dots, H_n\}$ be a topological line arrangement and fix $H \in \cA$.
Then, we have 
\[
\overline{\pi} (\cA,t) = 1 + n t + \sum_{P \in \Sing (\cA) \setminus \Sing_{H} (\cA)} (m_{P}-1)t^2.
\]
\end{proposition}

\proof
By the definition of the M\"obius function, we have 
\begin{eqnarray*}
\pi (\cA, t) &=&
1 + (n+1) t + \sum_{P \in \Sing (\cA)} (m_P -1) t^2 + \left (\sum_{P \in \Sing (\cA)}(m_P -1) -n  \right )t^3 \\
&=& (1+t) (1 + n t + \left ( \sum_{P \in \Sing (\cA)} (m_P -1 ) -n \right )t^2  ).
\end{eqnarray*}
Since $\sum_{P \in \Sing_{H}(\cA)}(m_{P}-1) = n$ for each $H \in \cA$, the claim follows.
\endproof

\subsection{Orlik-Solomon algebra}
Let $\cA = \{H_0 ,\dots, H_n\}$ be a topological line arrangement. 
Suppose that the topological lines are ordered as the labelling and fix $H_{0}$ as the ``line at infinity''.
For $P \in \Sing (\cA)$, we set $\min P =\min \{i \mid H_i \in \cA_{P}\}$.

\begin{definition}
A pair of positive integers $S=(j,k)$ ($1 \leq  j<k \leq n$) is called \textit{nbc}, if there exists $P \in \Sing (\cA)$ such that $P = H_{j} \cap H_{k}$ and $\min P = j$. 
We denote by $\mathbf{nbc}(\cA)$ the set of all nbcs in $\cA$.
\end{definition}

Note that the index $0$ does not appear in nbc. 
\begin{definition}
We define a graded $\bZ$-algebra $A=A(\cA) = \bigoplus_{k=0}^{2}A^{k}$ as follows: $A^{0} = \bZ$, $A^{1}$ is a free module with a basis $\{e_1, \ldots, e_n\}$, symbols corresponding to lines $H_1,\ldots, H_n$, and $A^2$ is defined as a free $\bZ$-module with a basis $\{f_{j,k} \mid (j,k) \in \mathbf{nbc}(\cA)\}$.
The multiplication $\mu: A^1 \wedge A^1 \rightarrow A^2$ is defined as:
\begin{equation*}
\mu(e_i, e_j) =
\left \{ 
\begin{array}{ll}
f_{i,j} & (\mbox{if $(i,j) \in \mathbf{nbc}(\cA)$}) , \\
f_{k,j} - f_{k,i} & (\mbox{if there exists $k$ such that $(k,i), (k,j) \in \mathbf{nbc}(\cA)$}) , \\
0, & (\mbox{otherwise})
\end{array}
\right .
\end{equation*}
for $1 \leq i < j \leq n$. This graded algebra $A$ is called the \textit{Orlik-Solomon algebra} of $\cA$.
\end{definition}

This definition coincides with the Orlik-Solomon algebra of an affine line arrangement obtained by the deconing (see \cite[Section 3.2]{orl-ter} and \cite[Section 7.10]{bjo} for general matroids).

By Proposition \ref{prop:exp_poin}, we have
\[
\sum_{k=0}^{2} (\rank_{\bZ}A^{k}) t^{k} = \overline{\pi} (\cA,t).
\]

\section{Homology groups}\label{sec:hom}
From now on, let $\cA = \{H_0,\dots, H_n \}$ be a topological line arrangement, and the lines are ordered as in this labelling. We fix $H_0$ as the ``line at infinity''.

In this section, we compute the homology groups of the complement and give an explicit basis.
We will consider (co)homology groups only with integer coefficients. Thus we will omit the coefficient.

\subsection{The topology of $V(\cA)$}
Let $V(\cA) = \bigcup_{i=0}^{n}H_{i}$ be the union of topological lines.
In this subsection, we study the topology of $V(\cA)$.
First, since $V(\cA)$ is a $2$-dimensional CW complex which is the union of $(n+1)$ transverse intersecting $2$-spheres, we have the following.

\begin{proposition}
$H_{2} (V(\cA)) \cong \bZ^{n+1}$ and $\{H_0,\dots, H_{n}\}$ is a basis of $H_2 (V(\cA))$.
\end{proposition}

However, $V(\cA)$ is not a wedge sum of $(n+1)$ spheres in general. In particular, $H_1(V(\cA))$ may be non-zero.

For each $H_{i}$, remove a small disk from it, and let $V^{\circ}(\cA)$ be the union of $H_{i}$'s with the disk removed. 
The Hasse diagram of the subset $L_{1} (\cA) \cup L_2(\cA)$ of the intersection poset can be considered as a bipartite graph, and denote it by $\Gamma (\cA)$. 
By taking a point $x_{i} \in H_{i}$ for each $i$ and connecting $x_{i}$ and $P \in \Sing (\cA)$ for each pair $P \in H_{i}$, we can embed $\Gamma(\cA)$ into $V^{\circ} (\cA)$.
It is easy to see that $V^{\circ}(\cA)$ deformation retracts to $\Gamma (\cA)$.
The graph $\Gamma (\cA)$ has the following topological property.

\begin{proposition}
The graph $\Gamma (\cA)$ is a connected graph with $b_1 (\Gamma (\cA)) = |\mathbf{nbc}(\cA)|$. Moreover, we can construct an explicit basis $\{\ell_{j,k}\}$ using nbc.
\end{proposition}

\proof
Let $\Gamma'(\cA)$ be a spanning tree whose edge set $T$ is defined by 
\[
T=\{(H, P) \mid P \in \Sing_{0} (\cA)\} \cup \{(H_{i}, P) \mid i = \min P\}.
\]
Then, the complement edge set $C$ is written as 
\[
C=\{(H_{i}, P) \mid P \notin \Sing_{0} (\cA) , \, i \neq \min P\}.
\]
Since $\Gamma (\cA)$ is a connected graph, $b_1 (\Gamma) = |C|$.
There exists a bijection between $C$ and $\mathbf{nbc} (\cA)$ defined by 
$(H_{i}, P) \mapsto (\min P, i)$. Thus we have $b_1 (\Gamma) = |\mathbf{nbc} (\cA)|$.
For nbc $(j,k)$ we define a cycle $\ell_{j,k}$ by concatenating six edges 
\[
H_0-H_{0} \cap H_j - H_{j} -H_{j} \cap H_{k} - H_{k} - H_{0} \cap H_{k} - H_{0}
\]
in $\Gamma (\cA)$. This cycle represents a basis of $H_1 (\Gamma (\cA))$.
\endproof

Since the inclusion $V^{\circ} (\cA) \to V(\cA)$ induces an isomorphism between the first homology groups, we have the following.
\begin{corollary}
$H_{1} (V(\cA)) \cong \bZ^{|\mathbf{nbc}(\cA)|}$ with the basis $\{\ell_{j,k}\}$.
\end{corollary}

\subsection{A regular neighborhood of an arrangement and its boundary}
Here, we recall the construction of the neighborhood of $V(\cA)$ and its boundary manifold; see also \cite{coh-suc-bound, fgm, eld, liu-xie}.

Let $\varepsilon_1, \varepsilon_2$ be sufficiently small numbers with $0 < \varepsilon_1 \ll \varepsilon_2 \ll 1$.
Let $N_{H} = N(H, \varepsilon_1)$ be the $\varepsilon_1$-neighborhood of $H$ for each $H \in \cA$, and $B_{P} = B(P,\varepsilon_2)$ be the $4$-ball centered at $P$ with the radius $\varepsilon_2$ for each $P \in \Sing (\cA)$. 
We orient these pieces by restricting that of $\bC P^2$.
The space $N_{H}$ is the total space of a $D^2$-bundle over $H$ with the Euler number $+1$. We denote by $\pi_{H} : N_H \to H$ the projection of this bundle.
The regular neighborhood $N(\cA)$ of the arrangement is described as
\[
N(\cA) = \bigcup_{H \in \cA} N_H \cup \bigcup_{P \in \Sing (\cA)} B_{P}.
\]
Note that the neighborhood $N(\cA)$ deformation retracts to $V(\cA)$.

The boundary manifold is obtained by the following construction.
For each $H_i \in \cA$, define the piece $X_{i}$ by 
\[
X_{i} = (\pi_{i} |_{\bd (N_H)})^{-1} (H_i \setminus \bigcup_{P \in \Sing_{i} (\cA)} (H_{i} \cap B_{P}))
\]
This is the total space of $S^1$-bundle over $H_i- \bigcup_{P \in \Sing_{i} (\cA)} (H_{i} \cap B_P)$, which is a genus zero surface with $n_i$ boundary components (recall that $n_i = |\Sing_i (\cA)|$). 
Thus the bundle is trivial and we can identify $X_{i}$ with $(S^2 \setminus \{\mbox{$n_i$ disks}\}) \times S^1$. This is a compact $3$-manifold with boundary. 
Each boundary component is a torus obtained by the product of the boundary of the base surface and $S^1$.

For each $P \in \Sing (\cA)$, define the piece $Y_P$ by 
\[
Y_P = \partial B_P \setminus \bigcup_{H \in \cA_{P}} \Int( N_H \cap \bd B_{P})
\]
The boundary $\partial B_P$ is a $3$-sphere and each intersection $N_H \cap \bd B_P$ is a solid torus with the core $K_{H,P} = H \cap \partial B_P$. 
$K_{H,P}$ is oriented as the boundary of the oriented disk $H \cap B_P$.
Let $\lambda_{H, P}$ and $\mu_{H, P}$ be the homology classes of the preferred longitude and meridian of the torus $T_{H}^{P} = \bd N_H \cap \bd B_P$.
The orientation of the meridian $\mu_{H,P}$ is induced by that of the meridian disk, which intersects transversally and positively with $H$. 
We give the orientation for the torus $T_H^P$ so that $(\lambda_{H,P}, \mu_{H,P})$ is a positive ordered basis. 
The piece $Y_{P}$ is a compact $3$-manifold obtained from a $3$-sphere by removing solid tori \footnote{Although the lines intersect one another transversally at a single point, this link $\{K_H \mid H \in \cA_P\}$ is not necessarily the $m_P$-component Hopf link in general. However, this is indeed the case in the symplectic and algebraic settings. }. 
Each boundary component of $Y_{P}$ is also a torus.
Since each component $H \in \cA$ intersects in exactly one point transversally and positively, $lk (K_{H,P}, K_{H',P}) =+1$ for any $H, H' \in \cA_{P}$, and each component $K_{H,P}$ is a trivial knot. 

The boundary manifold $\partial N$ is constructed by gluing $X_{i}$'s and $Y_{P}$'s. The boundary tori $\partial (H_{i} \cap B_P ) \times S^1 \subset X_{i}$ and $\partial N_{H_i} \cap \partial B_P$ are glued if $P \in H_{i}$.

\subsection{Homology groups of $U(\cA)$}
Here, we compute the homology group of $U(\cA)$ and give an explicit basis.
Let $U'(\cA) = \bC P^{2} \setminus \Int(N(\cA))$ be the exterior of the arrangement. $U'(\cA)$ is an orientable compact $4$-manifold with boundary, which is a deformation retract of $U(\cA)$.
Thus, we will work with $U'(\cA)$. Note that, $\bd U'(\cA) = - \bd N(\cA)$ as oriented manifolds.

\subsubsection{$H_1 (U'(\cA))$}
Recall that $H^{2}(V(\cA)) \cong \bZ^{n+1}$ and we can take the dual $[H_i]^{*}$ of each lines $H_i$ as a basis of $H^2 (V(\cA))$. By Poincar\'e duality and the excision isomorphism, we have 
\[
H^{2} (N(\cA)) \cong H_2 (N(\cA), \bd N(\cA)) \cong H_2 (\bC P^2 , U'(\cA)).
\]
Through this isomorphism, the dual $[H_{i}]^{*}$ of each line is mapped to the intersection dual $D_{i}$ of the line $H_{i}$. That is, $D_{i}$ is a $2$-relative cycle in $N(\cA)$ such that $\bd N(\cA) \cap D_{i} = \bd D_{i}$, and $[D_i] \cdot [H_{j}] = \delta_{i,j}$. 

It is easy to see that such a $2$-chain $D_{i}$ is realized as a meridian disk of each line $H_{i}$ (the fiber disk of the bundle $\pi_i: N_{H_i} \to H_{i}$ over a point).
We orient the disk $D_i$ so that $[D_i] \cdot [H_i] = +1$.
Let $m_{i} = \bd D_{i}$ and call it the \textit{meridian cycle} of the line $H_i$.
Note that $m_{i}$ is also a cycle in $U'(\cA)$, since it lies in the boundary.

\begin{theorem}\label{thm:homol_1st}
$H_1 (U'(\cA)) \cong \bZ^{n}$ and we can take $\{[m_i]\}_{i=1}^{n}$ as a basis.
\end{theorem}

\proof
By the homology long exact sequence for the pair $(\bC P^2, U'(\cA))$, we have the following exact sequence:
\[
H_2 (\bC P^2 ) \xrightarrow{j} H_2 (\bC P^2, U'(\cA)) \xrightarrow{\delta} H_1 (U'(\cA)) \to 0.
\]
$H_{2} (\bC P^2) \cong \bZ$ and we can take a generic complex line $L$ as its basis. Since the algebraic intersection number $[L] \cdot [H_i]$ is one for each line $H_{i}$, we have that $j([L]) = [D_0] + \dots + [D_n]$. Therefore, the above exact sequence becomes
\[
\bZ \xrightarrow{(1,1,\dots,1)} \bZ^{n+1} \xrightarrow{\delta} H_1 (U'(\cA)) \to 0.
\]
Moreover, by the definition of the connecting homomorphism, we have $\delta([D_{i}]) = [m_i]$. Thus we have 
\[
H_1 (U'(\cA)) \cong (\bZ m_0 \oplus \dots \oplus \bZ m_n ) / (m_0 + \dots + m_n).
\]
Thus, by substituting $m_0 = - (m_1 + \dots m_{n})$, we have the theorem.
\endproof

\subsubsection{$H_2 (U'(\cA))$}
We take a similar method to compute the second homology group.
By the Poincar\'e duality and excision, we have the isomorphism
\[
H^{1} (N(\cA)) \cong H_3 (N(\cA), \bd N(\cA)) \cong H_3 (\bC P^2, U'(\cA)).
\]
The dual $[\ell_{j,k}]^{*}$ of an nbc cycle is mapped to the intersection dual of $\ell_{j,k}$. 
For an nbc $(j,k)$, define the solid torus by $S_{j,k} = \bd B_P \cap N_{H_k}$. It is also easy to see that $S_{j,k}$ is the intersection dual of $\ell_{j,k}$ and $S_{j,k} \cap \bd N(\cA) = \bd S_{j,k}$. 
The torus $T_{j,k} = \bd S_{j,k}$ is called the \textit{nbc torus} of $(j,k)$. 
The nbc torus $T_{j,k}$ is oriented as in the previous subsection, that is, $(\lambda_{k,P}, \mu_{k,P})$ is the positive direction. 

\begin{theorem}
$H_2 (U'(\cA)) \cong \bZ^{|\mathbf{nbc}(\cA)|}$ and we can take $\{[T_{j,k}] \mid (j,k)\in \mathbf{nbc}(\cA)\}$ as a basis.
\end{theorem}

\proof
By the homology long exact sequence, we have 
\[
0 \to H_3 (\bC P^2, U'(\cA)) \xrightarrow{\delta} 
H_2 (U'(\cA)) \to H_2 (\bC P^2) \xrightarrow{j} H_2 (\bC P^2, U'(\cA)).
\]
By the proof of Theorem \ref{thm:homol_1st}, we have that $j$ is injective. Thus, we have the isomorphism
\[
\delta: H_3 (\bC P^2, U'(\cA)) \cong H_2 (U'(\cA)).
\]
This proves $H_2 (U'(\cA)) \cong \bZ^{|\mathbf{nbc}(\cA)|}$.
Since the connecting homomorphism satisfies $\delta ([S_{j,k}]) = [T_{j,k}]$, we obtain the description of the basis.
\endproof

We also have the following vanishing results for higher homology groups.
\begin{theorem}
For $k \geq 3$, $H_k (U'(\cA))=0$.
\end{theorem}

\proof
Since $U'(\cA)$ is a compact $4$-manifold with boundary, it follows that $H_{k} (U'(\cA)) = 0$ for $k \geq 4$ and $H_3 (U'(\cA))$ is a free module.
Moreover, we have the following exact sequence:
\[
0 \to H_4 (\bC P^2) \xrightarrow{j} H_4 (\bC P^2, U'(\cA)) \to H_3 (U'(\cA)) \to H_3 (\bC P^2)=0.
\]
Now, we have $H_4 (\bC P^2 , U'(\cA)) \cong H^{0} (N(\cA))=\bZ$ and $j$ is injective.
Since $H_3 (U'(\cA))$ is a free module, it follows that $H_3 (U'(\cA))=0$.
\endproof

To summarize, we have the following formula for the Betti numbers of the complement.
\begin{corollary}\label{cor:poincarepoly}
\[
\overline{\pi}(\cA,t) = \sum_{i \geq 0} b_i (U(\cA))t^{i}.
\]
\end{corollary}

\section{Cohomology rings}\label{sec:cohom}
In this section, we compute the cohomology ring and prove that it is isomorphic to the Orlik-Solomon algebra.
Here, we also work with the exterior $U'(\cA)$.
We give two ways to compute the cohomology ring; the first is by the intersection product in relative homology, and the second is by direct computations.

\subsection{A proof by the homology intersection product.}
Here, we compute the cohomology ring by the homology intersection product. 
Instead of computing the cohomology ring directly, we will compute the homology intersection ring $H_{4-*} (U', \bd U')$ (see Section \ref{sec:app}) here.
Since the non-trivial cup product is $H^1 (U') \times H^{1} (U') \to H^2 (U')$, it suffices to compute the intersection product $H_3 (U', \bd U') \times H_3 (U', \bd U') \to H_2 (U', \bd U')$.

\subsubsection{Basis of $H_3 (U', \partial U')$}
We will construct the basis of $H_3 (U', \partial U')$ as the intersection dual of $H_1 (U')$. Precisely, we will construct a compact $3$-manifold with boundary $G_{i}$ such that 
\begin{itemize}
\item $G_{i} \cap \partial U' = \partial G_{i}$, and 
\item $[G_{i}] \cdot [m_{j}] = \delta_{i,j}$.
\end{itemize}
Then, $\{[G_1], \dots, [G_n]\} \subset H_3 (U', \partial U')$ forms a basis of $H_3 (U', \partial U')$.

First, we will construct a closed surface $F_{i}$ embedded in $\partial U'$, which will be the boundary $\partial G_{i}$ of the relative cycle.
The surface $F_{i}$ is constructed from small pieces in each $X_{i}$ and $Y_{P}$, and by gluing them step by step as in Figure \ref{fig:step}.

\begin{figure}[htbp]
\centering
\begin{tikzpicture}
\coordinate (A) at (0,0);

\draw (A)++(-0.5,0) --++(5,0)node[right]{$H_i$};
\draw (A)++(-0.2,-0.3) --++(2.4,3.6) ++ (0.2,0)node[above]{$H_j$};
\draw (A) ++(4,0)++(0.2,-0.3) --++(-2.4,3.6) ++ (-0.2,0)node[above]{$H_0$};

\draw (A) circle (0.2);
\draw (A)++(4,0) circle (0.2);
\draw (A) ++ (2,3) circle (0.2);

\draw[<-] (A) ++ (2,0)--++(0,-0.5)node[below]{(1) $F_{i,1,+}$};
\draw[<-] (A) ++ (3,1.5)--++(0.5,0)node[right]{(2) $F_{i,1,-}$};
\draw[<-] (A) ++ (3.8,0.1)to[out=130,in=180]++(0.8,0.5)node[right]{(3) $F_{i,2,P_0}$};
\draw[<-] (A) ++ (0.2,0.1)to[out=50,in=0]++(-0.8,0.5)node[left]{(4) $F_{i,2,P}$};
\draw[<-] (A) ++ (2,2.8)to[out=-90,in=-90]++(-1,0)node[above]{(5) $F_{i,2,P}$};
\draw[<-] (A) ++ (1,1.5)--++(-0.5,0)node[left]{(6) $F_{i,3,j}$};

%\draw (A)++(0.1,0) ++ (0.3,0.5) --++(1,1.5);

\end{tikzpicture}
\caption{Schematic picture of the construction of $F_i$}
\label{fig:step}
\end{figure}
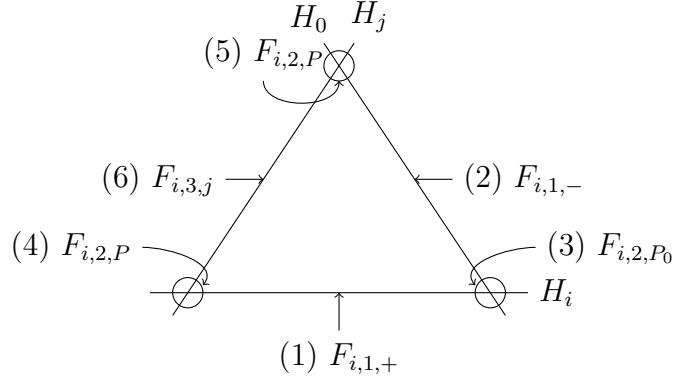

\begin{enumerate}[(1)]
\item 
\underline{In the piece $X_i$.}
Recall that $X_{i}$ is the total space of a trivial $S^1$-bundle over 
$H_{i} \setminus \bigcup_{\Sing_{i}(\cA)} (B_P \cap H_{i})$.
We take a section of this bundle, and denote it by $F_{i,1,+}$.
The piece $F_{i,1,+}$ is a compact genus-zero surface with $n_{i}$ boundary components. 
Since the Euler number of the bundle $\pi_{i}$ is $+1$, the section can be taken so that each boundary component is represented by $\lambda_{i, P}$ in $Y_P$, which is parallel with a circle $\partial (B_P \cap H_{i})$ in the boundary torus $\partial (B_P \cap H_{i}) \times S^1$ if $P \neq P_0$, and the boundary component in $\bd B_P$ is a cycle represented by $\lambda_{i, P_0} + \mu_{i, P_0}$ in $Y_P$. 
Here, we give the orientation of $F_{i,1,+}$ so that $[F_{i,1,+}] \cdot [m_{i}] = +1$.

\item 
\underline{In the piece $X_0$.}
Similarly, we take a section of the trivial $S^1$-bundle $X_{0} \to H_0 \setminus \bigcup_{P \in \Sing_{0} (\cA)} (B_P \cap H_0)$, and denote it by $F_{i,1,-}$.

\item 
\underline{In the piece $Y_{P_0}$.}
Let $P_0 = H_{i} \cap H_0$. By the steps (1) and (2), the circle $Y_{P_0} \cap F_{i,1,+}$ is represented by $\lambda_{i, P_0} + \mu_{i, P_0}$ in the torus $T_{i}^{P_0}$.
Similarly, $Y_{P_0} \cap F_{i,1,-}$ is represented by $\lambda_{0,P_0} + \mu_{0,P_0}$.
For any pair $H$ and $H'$ in $\cA_{P_0}$, the linking number $lk (K_{H, P_0} \cap K_{H', P_0})$ is $+1$ in $\bd B_P$. 
Therefore, by definition of the linking number, we have 
\[
\lambda_{H,P_0} + \mu_{H,P_0} = \left ( \sum_{H' \in \cA_{P_0}, H' \neq H} lk (K_{H,P_0},  K_{H',P_0})\mu_{H',P_0} \right ) + \mu_{H,P_0} = \sum_{H' \in \cA_{P_0}} \mu_{H',P_0} 
\]
for any $H \in \cA_{P_0}$. In particular, $\lambda_{i,P_0} + \mu_{i,P_0}$ and $\lambda_{0,P_0} + \mu_{0,P_0}$ represent the same homology class in $H_{1} (Y_{P_0}; \bZ)$. Therefore, there exists a compact surface $F_{i,2,P_0}$ in $Y_{P_0}$ such that $\bd F_{i,2,P_0} = (Y_{P_0} \cap F_{i,1,+}) \cup (- (Y_{P_0} \cap F_{i,1,-}))$ (as an oriented manifold).

Therefore, for the glued surface $F_{i,1,+} \cup F_{i,2, P_0} \cup F_{i,1,-}$ to admit an orientation compatible with that of $F_{i,1,+}$, the piece $F_{i,1,-}$ must be given the orientation opposite to its usual one, so that $F_{i,1,-} \cdot m_{0} =-1$.

\item 
\underline{In the piece $Y_P$ ($P \in \Sing_{i} (\cA) \setminus \{P_0\}$).}
Let $P \in \Sing_{i} (\cA) \setminus \{P_{0}\}$. 
The intersection $Y_P \cap F_{i,1,+}$ is the longitude $\lambda_{i,P}$ of the torus $T_{i}^{P}$ in $Y_P$.
Since $K_{i}$ is a trivial knot, there exists a disk $D_{i,P}$ in $\bd B_P \setminus N_{H_i}$ whose boundary is represented by $\lambda_{i,P}$.
The disk $D_{i,P}$ is oriented by extending that of $F_{i,1,+}$.
Let $F_{i,2,P} = D_{i,P} \cap Y_{P}$. 
The piece $F_{i,2, P}$ is a compact genus-zero surface with boundary. One of the boundary components is represented by the longitude $\lambda_{i, P}$. 
Another boundary component is described as $\partial (D_{i,P} \cap N_{H_j})$ for $H_{j} \neq H_{i}$. Let $c_{j,i,P} = \partial (D_{i,P} \cap N_{H_j})$.
This is (possibly the disjoint union of) meridian cycles around $H_j$ in $Y_{P}$. 
However, since $lk (K_{i,P}, K_{j,P})=+1$, $[c_{j,i,P}] = +\mu_{j,P}$ in $H_{1} (Y_P, \bZ)$.
The boundary of $F_{i,2,P}$ is described as $\partial F_{i,2,P} = \lambda_{i,P} \cup \bigcup_{H_{j} \in \cA_{P} \setminus \{H_{i}\}}c_{j,i,p}$.
See Figure \ref{fig:i2p}.

\item 
\underline{In the piece $Y_P$ ($P \in \Sing_{0} (\cA) \setminus \{P_0\}$).}
We perform an analogous construction for each $P \in \Sing_{0} (\cA) \setminus \{P_0\}$, and obtain a surface $F_{i,2,P}$ ($P\in \Sing_{0} (\cA)$) whose boundary is $\partial F_{i,2,P} = \lambda_{0,P} \cup \bigcup_{H_{j} \in \cA_{P} \setminus \{H_{0}\}} c_{j,0,P}$.

\item 
\underline{In the piece $X_j$.}
The union of the pieces obtained above 
\[
F_{i,1,+} \cup F_{i,1,-} \cup F_{i,2,P_0} \cup \bigcup_{P \in \Sing_{i}(\cA)\setminus \{P_0\}} F_{i,2,P} \cup \bigcup_{P \in \Sing_{0}(\cA) \setminus \{P_0\}} F_{i,2,P}
\]
is a compact surface with the boundary $c_{j,i,P}$ ($P_{0} \notin H_{j}$). This is because the surfaces of $F_{1,*}$ and $F_{2,*}$ are glued along the longitude cycle $\lambda$'s.

Here, we set $P=H_i \cap H_j$ and $Q= H_0 \cap H_j$. In $X_{j}$, each component of $c_{j, i, P}$ is a meridian cycle, that is a circle homologous to the fiber of the $S^1$-bundle $\pi_{j}$ over some point in $\bd (H_{j} \cap B_{H_i \cap H_j})$. 
$\pi_j (c_{j, i, P})$ is a finite point set and since $[c_{j, i, P}] = + \mu_{j, P}$ in the homology class, we have that the signed sum of all the points in $\pi_{j} (c_{j, i, P})$ is $+1$, in other words, $[\pi_{j} (c_{j, i, P})] = +1 \in H_0 (H_{j};\bZ)$. 
Similarly, since $F_{i,1,-}$ has the opposite orientation, we have that the sum of signs in $\pi_{j} (c_{j,0, Q}) = -1$. 
Therefore the sum $[\pi_{j} (c_{j,i,P})] + [\pi_{j}(c_{j,0,Q})] = 0$ in $H_0 (H_{j}; \bZ)$. Thus,
there exists a compact oriented $1$-dimensional manifold $\Gamma_{i,j}$ such that $\bd (\Gamma_{i,j}) = \pi_{j} (c_{j,i,P}) \cup \pi_{j}(c_{j,0,Q})$. 
Note that $\Gamma_{i,j}$ may contain a circle component and when $H_{j} \in \cA_{P_0}$, $\Gamma_{i,j}$ is a closed manifold. 
Finally, we define the piece by $F_{i,3,j} = \pi_{j}^{-1} (\Gamma_{i,j})$. Then, $\bd F_{i,3,j} = c_{j,i,P} \cup c_{j,0,Q}$. 
By gluing this piece to the above piece, we can close the boundary components.
Finally, we denote the resulting surface by $F_{i}$, that is,
\[
F_{i} = F_{i,1,+} \cup F_{i,1,-} \cup F_{i,2,P_0} \cup \bigcup_{P \in \Sing_i (\cA)\setminus \{P_0\}} F_{i,2,P} \cup \bigcup_{P \in \Sing_{0} (\cA)\setminus \{P_0\}} F_{i,2,P} \cup \bigcup_{j \neq i} F_{i,3,j}.
\]
We give the orientation of the surface $F_{i}$ by extending it in the piece $F_{i,1,+}$.
\end{enumerate}

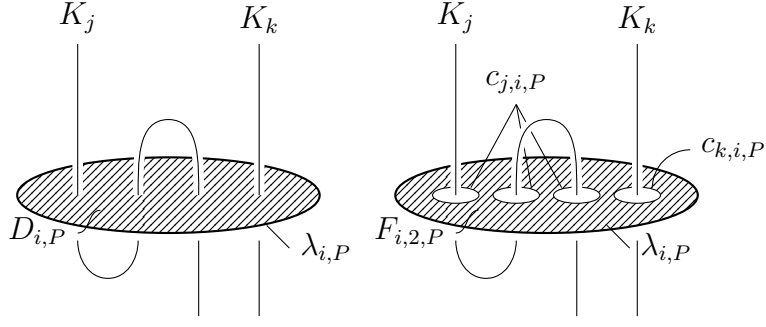
\begin{figure}[htbp]
\centering
\begin{tikzpicture}
\coordinate (A) at (0,0);
\coordinate (B) at (5,0);

\draw[thick] (A) ellipse (2 and 0.5);
\fill[pattern=north east lines] (A) ellipse (2 and 0.5);
\draw[preaction={draw=white,line width=4pt}] (A) ++ (1.2,2) --++(0,-2);
\draw[preaction={draw=white,line width=4pt}] (A) ++ (-1.2,2) --++(0,-2);
\draw[preaction={draw=white,line width=4pt}] (A)++ (-0.4,0) to[out=90,in=180]++(0.4,1) to[out=0,in=90]++(0.4,-1);
\draw (A)++(1.2,-0.6) --++(0,-1);
\draw (A)++(-1.2,-0.6) to[out=-90,in=180]++(0.4,-0.5) to[out=0,in=-90]++(0.4,0.5);
\draw (A)++(0.4,-0.6) --++(0,-1);

\draw (A)++(-1.2,2) node[above]{$K_{j}$};
\draw (A)++(1.2,2) node[above]{$K_{k}$};
\draw (A) ++ (1.3,-0.4)--++(0.3,-0.3) node[right]{$\lambda_{i,P}$};
\draw (A) ++ (-0.9,-0.2) to[out=180,in=0] ++(-0.3,-0.3)node[left]{$D_{i,P}$};

\draw[thick] (B) ellipse (2 and 0.5);
\fill[pattern=north east lines] (B) ellipse (2 and 0.5);
\draw [thick](B) ++ (1.2,0) ellipse (0.3 and 0.1);
\draw [thick](B) ++ (0.4,0) ellipse (0.3 and 0.1);
\draw [thick](B) ++ (-0.4,0) ellipse (0.3 and 0.1);
\draw [thick](B) ++ (-1.2,0) ellipse (0.3 and 0.1);
\fill[white] (B) ++ (1.2,0) ellipse (0.3 and 0.1);
\fill[white] (B) ++ (0.4,0) ellipse (0.3 and 0.1);
\fill[white] (B) ++ (-0.4,0) ellipse (0.3 and 0.1);
\fill[white] (B) ++(-1.2,0) ellipse (0.3 and 0.1);

\draw (B) ++ (-0.4,1.2) --++(-0.6,-1.1);
\draw (B) ++ (-0.4,1.2) --++(0.2,-1.1);
\draw (B) ++ (-0.4,1.2) --++(0.6,-1.1);
\draw[preaction={draw=white,line width=4pt}] (B) ++ (1.2,2) --++(0,-2);
\draw[preaction={draw=white,line width=4pt}] (B) ++ (-1.2,2) --++(0,-2);
\draw[preaction={draw=white,line width=4pt}] (B)++ (-0.4,0) to[out=90,in=180]++(0.4,1) to[out=0,in=90]++(0.4,-1);
\draw (B)++(1.2,-0.6) --++(0,-1);
\draw (B)++(-1.2,-0.6) to[out=-90,in=180]++(0.4,-0.5) to[out=0,in=-90]++(0.4,0.5);
\draw (B)++(0.4,-0.6) --++(0,-1);

\draw (B)++(-1.2,2) node[above]{$K_{j}$};
\draw (B)++(1.2,2) node[above]{$K_{k}$};
\draw (B) ++ (0.8,-0.4)--++(0.3,-0.3) node[right]{$\lambda_{i,P}$};
\draw (B) ++ (-0.9,-0.2) to[out=180,in=0] ++(-0.3,-0.3)node[left]{$F_{i,2,P}$};
\draw (B) ++ (1.4,0.1) to[out=90,in=180] ++(0.5,0.5)node[right]{$c_{k,i,P}$};
\draw (B) ++ (-0.4,1.2) node[above]{$c_{j,i,P}$};

\end{tikzpicture}
\caption{The piece $F_{i,2,P}$. In this figure, $P = H_i \cap H_j \cap H_k$.}
\label{fig:i2p}
\end{figure}

\begin{remark}
Precisely, the choice of each $1$-manifold $\Gamma_{i,j}$ is defined later with the construction of the manifold $G_{i}$ whose boundary is $F_{i}$.
\end{remark}

\begin{proposition}\label{prop:dual_bd}
The closed surface $F_{i}$ satisfies $F_{i} \cdot m_{j} = \delta_{i,j}$.
\end{proposition}

\proof
In the piece $X_i$, the meridian $m_i$ is seen as the fiber of the $S^1$-bundle over some point. By step (1), $F_{i,1,+}$ is a section of this bundle. Thus, $F_{i}$ intersects $m_i$ transversally at once.
If any other meridians appear, then they can appear only in the piece $X_{j}$ in step (6). 
However, we can choose the center of the meridian disk of $H_{j}$ so that it does not lie on the $1$-manifold $\Gamma_{i,j}$. With this choice, the piece $F_{i,3,j}$ does not meet the meridian $m_{j}$.
\endproof

Next, we will construct the $3$-manifold $G_{i}$ in $U'(\cA)$ with $\partial G_{i} = F_{i}$.

Let us define the piece
$S_{i,+} := F_{i,1,+} \cup \bigcup_{P \in \Sing_{i}(\cA) \setminus \{P_0\}} D_{i,P}$. 
This piece $S_{i,+}$ is a disk embedded in $X= \bC P^2 \setminus \Int (B_{P_0})$ with the boundary represented by $\lambda_{i, P_0} + \mu_{i,P_0}$. Similarly, we define 
$S_{i,-} = F_{i,1,-} \cup \bigcup_{P \in \Sing_{0}(\cA) \setminus \{P_0\}} D_{0,P}$.
Connecting $S_{i,+}$ and $S_{i,-}$ by a surface $F_{i,2,P_{0}}$, we obtain an embedded closed surface in $X$. Moreover, by pushing this surface into $\Int (X)$ along the collar neighborhood of the surface $F_{i,2,P_0}$, we have a closed surface $S_{i}=S_{i,+} \cup F_{i,2,P_0} \cup S_{i,-}$ embedded in $\Int (X)$.  
Similarly to $F_{i,1,-}$, around $S_{i,-} \subset S_{i}$, it has the orientation opposite to the usual one such that $S_{i,-} \cdot m_{0} =-1$.

The surface $S_i$ represents a second homology class $[S_i] \in 
H_2 (X; \bZ)$.
Since $X$ is the complement of a $4$-ball in $\bC P^2$, $X$ is the total space of the $D^2$-bundle over $\bC P^1$ with the Euler number $+1$.
The base space can be taken as a complex line $L \cong \bC P^1$ that does not pass through $B_{P_0}$.
Now, the intersection number is
\[
[S_{i}] \cdot [L] = [S_{i,+}] \cdot [L] + [F_{i,2,P_0}] \cdot [L] + [S_{i,-}] \cdot [L] 
= 1+0 -1 =0.
\]
Thus, $[S_i] =0 \in H_2 (X; \bZ)$. 
Therefore, there exists a compact orientable $3$-manifold $G'_{i}$ in $X$ such that $\bd G'_i = S_i$ (see \cite[Proposition 27.5]{rani}).

This manifold $G'_i$ may intersect other lines $H \in \cA$, that is $G'_{i}$ may not be a $3$-manifold \textit{in} $U'(\cA)$. 

We can generically assume that $G'_{i}$ and each line $H_j \in \cA$ intersect transversally (see \cite{gul-pol} and \cite[Chapter 9]{fre-qui} for a locally flat version). Each intersection $G'_{i} \cap H_j$ is a $1$-dimensional manifold. Let us denote this $1$-manifold by $\Gamma_{i,j} = G'_{i} \cap H_j$, which will play the role of the piece that appeared in step (vi).  
Moreover, by a small perturbation if necessary, we can assume that $\Gamma_{i,j} \subset H_j \setminus \bigcup_{P \in \Sing_{j} (\cA)} (H_j \cap B_P)$. Since $H_j$ is a closed manifold, we have
$\bd (\Gamma_{i,j}) = \bd G'_i \cap H_j$.
Since $\bd G'_{i} = S_{i,+} \cup F_{i,2,P_0} \cup S_{i,-}$ and by definition, 
\begin{eqnarray*}
\bd G'_{i} \cap H_j = 
\left \{ 
\begin{array}{ll}
(D_{i,H_{i} \cap H_j} \cap H_j) \cup (D_{0,H_0 \cap H_j} \cap H_j) & (H_j \notin \cA_{P_0}), \\
\emptyset & (H_j \in \cA_{P_0}).
\end{array}
\right .
\end{eqnarray*}
Thus, if $H_j \notin \cA_{P_0}$, then the intersection $\Gamma_{i,j}$ contains a $1$-manifold with boundary, and their boundary points are lying in 
$D_{i,H_{i} \cap H_j} \cap H_j$ and $D_{0,H_0} \cap H_j$. 

Finally, by removing tubular neighborhood $\pi_{j}^{-1} (\Gamma_{i,j})$ of $\Gamma_{i,j}$ which is identified with $\cong \Gamma_{i,j} \times D^2$ from $G'_i$ for each $j$, we obtain a $3$-manifold $G_{i}$. 
We orient $G_i$ so that the boundary orientation induced on $\bd G_{i}$ agrees with the orientation of $F_{i}$ fixed before.

\begin{proposition}
The manifold $G_{i}$ satisfies $G_{i} \cap \bd U' = \bd G_i$ and $G_{i} \cdot m_{j} = \delta_{i,j}$, that is, $[G_{i}] \in H_3 (U', \bd U')$ is the intersection dual of $m_{i}$.
\end{proposition}

\proof
Since we have 
\[
\pi_{j}^{-1}(\Gamma_{i,j}) \cap \partial G'_{i} = \pi^{-1}_{j} (\pi_{j} (c_{j,i,P}) ) = \bigcup_{j} (\partial \Gamma_{i,j} \times D^2),
\]
the boundary of $G_{i}$ satisfies
\[
\partial G_{i} = ( S_{i} \setminus \bigcup_{j} (\partial \Gamma_{i,j} \times D^2) ) \cup \bigcup_{j} (\Gamma_{i,j} \times S^1), 
\]
(the piece $\Gamma_{i,j} \times S^1$ corresponds to the piece $F_{i,3,j}$ which appeared in step (6)). 
The piece $\Gamma_{i,j} \times S^1$ is contained in $X_{j}$. Thus, $\bd G_{i}$ is contained in $\bd U'$. Moreover, if necessary, we can perturb $G_{i}$ so that its interior does not intersect with $\bd U'$. Thus, $G_{i}$ defines a relative $3$-cycle $[G_i] \in H_{3} (U', \bd U')$. 

Since the meridian cycle is contained in $\bd U'$ and the outward normal first convention, $G_i \cdot m_j = \bd G_{i} \cdot m_{j}$. By Proposition \ref{prop:dual_bd}, we have $\bd G_{i} \cdot m_{j} = \delta_{i,j}$. Thus $G_{i}$ is the intersection dual.
\endproof

\subsubsection{Basis of $H_2 (U', \partial U')$}
Before constructing the basis of the second relative homology group, we now examine the local intersection signs used below.
Let $P = H_i \cap H_j$. Since $H_i$ and $H_j$ intersect transversally and positively at $P$, 
we may choose an orientation-preserving local identification around $B_P$ with complex coordinates $(z_i, z_j)$ such that $H_i = \{z_i = 0\}$ and $H_j = \{z_j=0\}$.
We also give polar coordinates $z_i = r_i e^{\sqrt{-1}\theta_i}$ and $z_j = r_j e^{\sqrt{-1} \theta_j}$.
The positive directions of the meridians $\mu_{i,P}$ and $\mu_{j,P}$ are given by $\theta_i$ and $\theta_j$, respectively. We orient the tori $T_i^{P}$ and $T_j^P$ so that $(\theta_j, \theta_i)$ and $(\theta_i, \theta_j)$ are positive ordered bases, respectively.
Note that this agrees with the orientation defined so far, such that $(\lambda_{H,P}, \mu_{H,P})$ is the positive direction for $T_{H}^{P}$.
Now, the disk $D_{i,P}$ is oriented so that $\partial D_{i,P} = K_{i,P}$ as oriented manifolds. 
Thus, since $lk (K_{i,P}, K_{j,P}) = +1$, we have $[D_{i,P}] \cdot [K_{j,P}] = +1$. 
By the transverse intersection orientation convention, we have $([F_{i,2,P} \cap F_{j,2,P}]) \cdot [T_{j}^P] =([D_{i,P} \cap D_{j,P}]) \cdot [T_j^P]= +1$ and similarly $([F_{i,2,P} \cap F_{j,2,P}]) \cdot [T_{i}^P] = -1$.

Let us turn to the construction of the basis of the relative homology. After a small perturbation, we may assume that the relative $3$-cycles $G_{i}$ intersect transversally. We define
\[
M_{i,j} = G_i \cap G_{j},
\]
and give it the standard transverse intersection orientation (note that $M_{i,j} = -M_{j,i}$ as oriented manifolds). 

\begin{proposition}\label{prop:dual_2nd}
$M_{i,j}$ is the intersection dual of $T_{i,j}$. Thus, $\{[M_{i,j}] \mid (i,j) \in \mathbf{nbc} (\cA)\}$ is a basis of $H_2 (U', \bd U')$.
\end{proposition}

\proof
Since $M_{i,j} \cap \bd U' = \bd M_{i,j} = F_i \cap F_j$ and each nbc torus is contained in $\bd U'$, it suffices to examine the intersection of the boundaries. 
In particular, $M_{i,j} \cdot T_{k,l} = \bd M_{i,j} \cdot T_{k,l}$ by the outward normal first convention.
Each nbc torus is contained in a sphere $\bd B_P$ for some $P \in \Sing (\cA) \setminus \Sing_0 (\cA)$. 
Thus, to compute the intersection product, it suffices to consider the intersection around these small spheres.
By step (4) in the construction, if $P \in \Sing (\cA) \setminus \Sing_0 (\cA)$ is such that $\bd B_P$ intersects $F_{i}$, then $P \in \Sing_i (\cA) \setminus \{P_0\}$. Similarly, if $P \in \Sing (\cA) \setminus \Sing_0 (\cA)$ is such that $\bd B_P$ intersects $F_{j}$, then $P \in \Sing_j (\cA) \setminus \{P_0\}$. 
Thus, the only point $P$ satisfying both conditions is $P=H_i \cap H_j$. 
In particular, $[M_{i,j}] \cdot [T_{k,l}] = 0$ if $H_k \cap H_l \neq P$.
By the construction, the intersection of $F_i$ and $F_j$ in $\bd B_P$ is $F_{i,2,P} \cap F_{j,2,P}$. 
From the observation on the orientation above, we have that $(F_{i,2, P} \cap F_{j,2, P}) \cdot T_{i,j} = +1$. 
Let $H_k$ be another line such that $(i,k)$ is an nbc. Then, both intersections $F_{i,2,P} \cap T_{i,k}$ and $F_{j,2,P} \cap T_{i,k}$ are homologous to the meridian of $K_{k}$. 
Thus their algebraic intersection number is zero.
Therefore, $(F_{i,2,P} \cap F_{j,2,P}) \cdot T_{i,k} = (F_{i,2,P} \cap T_{i,k}) \cdot 
(F_{j,2,P} \cap T_{i,k}) = 0$. 
\endproof

\begin{figure}[htbp]
\centering
\begin{tikzpicture}
\coordinate (A) at (0,0);
\draw[thick] (A) ellipse (1.5 and 0.5);
\fill[pattern=north east lines] (A) ellipse (1.5 and 0.5);
\draw[preaction={draw=white,line width=4pt}, thick] (A)++(-1.5,0)ellipse (1 and 1.5); 
\fill[pattern=north west lines] (A)++(-1.5,0) ellipse (1 and 1.5);
\draw[
    preaction={draw=white,line width=3pt}, thick
]
plot[domain=180:255]
({1.5*cos(\x)},{0.5*sin(\x)});

\filldraw (A)++(-0.5,0) circle (0.06);
\filldraw (A)++(-1.5,0) circle (0.06);
\draw[very thick] (A)++(-1.5,0) --++(1,0);
\draw[very thick, ->] (A) ++ (-1.1,0)--++(0.15,0);

\draw (A)++(0,-0.5)--++(0,-0.3) node[below]{$\lambda_{i,P}$};
\draw (A)++(-2.5,-0)--++(-0.3,0) node[left]{$\lambda_{j,P}$};
\draw (A) ++ (0.7,0.2) to[out=90,in=-90] ++(0.3,0.5)node[above]{$F_{i,2,P}$};
\draw (A) ++ (-2,0.7) to[out=180,in=0] ++(-0.5,0.3)node[left]{$F_{j,2,P}$};

\end{tikzpicture}
\caption{The intersection $F_{i,2,P} \cap F_{j,2,P}$}
\label{fig:hopf}
\end{figure}
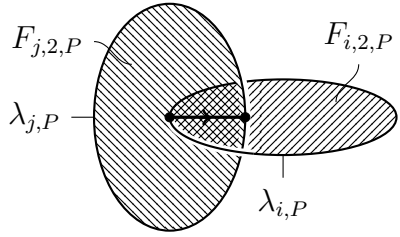

%\subsection{Computation of the intersection product}
%In this subsection, we compute the intersection product of the basis constructed above and prove that the intersection ring is isomorphic to the Orlik-Solomon algebra.

\subsubsection{Computation of the intersection product}
Under this setting, we can compute the intersection product of the relative cycles constructed above. 

\begin{theorem}\label{thm:ringstr}
The intersection product $\cdot : H_3 (U', \bd U') \times H_3 (U', \bd U') \to H_2 (U', \bd U')$ satisfies
\begin{eqnarray*}
[G_i] \cdot [G_j] = 
\left \{
\begin{array}{ll}
[M_{i,j}] & (\mbox{if $(i,j) \in \mathbf{nbc}(\cA)$}) , \\
{}[M_{k,j}] - [M_{k,i}]   & (\mbox{if there exists $k$ such that $(k,i), (k,j) \in \mathbf{nbc}(\cA)$}) , \\
0 & (\mbox{otherwise}).
\end{array}
\right .
\end{eqnarray*}
\end{theorem}

We prove this theorem by considering the three cases separately.
\begin{proposition}\label{prop:prod_nonnbc}
If there exists $k$ such that $(k,i), (k,j) \in \mathbf{nbc}(\cA)$, then $[G_i] \cdot [G_j] = [M_{k,j}] - [M_{k,i}] $ in $H_2 (U' , \bd U')$.
\end{proposition}

\proof
By the observation on the orientation before, we have 
$([G_{i}] \cdot [G_{j}]) \cdot [T_{k,j}] = [M_{i,j}] \cdot [T_{j}^P] = (F_{i,2,P} \cap F_{j,2,P}) \cdot [T_{j}^{P}] = +1$. Similarly, $([G_{i}] \cdot [G_{j}]) \cdot [T_{k,i}] = -1$.
By a similar argument in the proof of Proposition \ref{prop:dual_2nd}, it follows that the intersection product of $[G_i] \cdot [G_j]$ with other nbc tori vanishes. 
Therefore, by the duality, we have $[G_i] \cdot [G_j] = [M_{k,j}]- [M_{k,i}]$.
\endproof

\begin{proposition}\label{prop:prod_infty}
Suppose that $H_{i} \cap H_{j} \in H_{0}$. Then, $[G_i] \cdot [G_j] = 0$.
\end{proposition}

\proof
By the duality, it suffices to show that $([G_i] \cdot [G_j]) \cdot T_{k,l} = 0$ for each nbc $(k,l) \in \mathbf{nbc}(\cA)$.
Since $H_{i} \cap H_{j} \in H_{0}$, $\Sing_{i} (\cA) \cap \Sing_{j} (\cA) = \{H_{i} \cap H_{j}\}$.
In particular, $(\Sing_i (\cA) \setminus \{P_0\}) \cap (\Sing_j (\cA) \setminus \{P_0\}) = \emptyset $.
Thus, the intersection does not appear in the pieces in the step $(4)$. 
Each nbc torus is contained in the sphere $\bd B_P$ for some $P \in \Sing (\cA) \setminus \Sing_0 (\cA)$. 
Thus, the intersection $G_i \cap G_j$ does not intersect with any nbc torus. 
\endproof

\textit{Proof of Theorem \ref{thm:ringstr}}.
If $(i,j) \in \mathbf{nbc}(\cA)$, then it is clear that $[G_i] \cdot [G_j] 
= [M_{i,j}]$ by definition. Therefore, by Proposition \ref{prop:prod_nonnbc} and \ref{prop:prod_infty}, the theorem follows.

\endproof

\textit{Proof of Theorem \ref{thm:main_cohom}}. 
By the Poincar\'e-Lefschetz duality, the cohomology ring is isomorphic to the intersection ring $H_{4-*} (U', \bd U')$. 
By Theorem \ref{thm:ringstr}, the map $\Phi: H_{4-*} (U', \bd U') \to A^{*}(\cA)$ defined by $[G_{i}] \mapsto e_{i}, [M_{j,k}] \mapsto f_{j,k}$ is an isomorphism of algebras.
\endproof

\subsection{A proof by direct computations}
In this subsection, we give a simpler proof for the Orlik-Solomon description of the cohomology ring by a direct computation.

As we saw in Section \ref{sec:hom}, the homology group of the complement is a free module. Thus, by the universal coefficient theorem, the cohomology group satisfies 
\[
H^{k} (U;\bZ) \cong \Hom (H_k (U;\bZ), \bZ)
\]
for each $k \geq 0$. Let $\{\alpha_{i}\}_{i=1,\dots, n} \subset H^{1} (U; \bZ)$ be the dual basis of meridian cycles and $\{\beta_{r,s}\}_{(r,s) \in \mathbf{nbc}(\cA) } \subset H^2 (U;\bZ)$ be the dual basis of nbc tori.
To compute the cup product $\alpha_i \cup \alpha_j$, it suffices to compute the pairing with each nbc torus for each $i,j$. 

Let us fix an nbc $(r,s)$.
As we saw before, we can take the meridian $\mu_{r,s}$ and longitude $\lambda_{r,s}$ such that $[\mu_{r,s}] = m_s$ and $[\lambda_{r,s}] = \sum_{j \neq s, H_j \in \cA_{P}}lk(K_j,K_s)m_{j} =  \sum_{j \neq s, H_{j \in \cA_P }}m_{j} $, where $P = H_r \cap H_s$.
Let $\mu_{r,s}^{*}$ and $\lambda_{r,s}^{*}$ be their duals in $H^{1} (T_{r,s}; \bZ)$.
Now, the torus $T_{r,s}$ is oriented so that $(\lambda_{r,s}, \mu_{r,s})$ is a positive orientation. Thus, we have $\langle \lambda_{r,s}^{*} \cup \mu_{r,s}^{*}, [T_{r,s}]\rangle = 1$. 

To compute the cup product, we set $\alpha_{i} 
= a_{i,\lambda_{r,s}} \lambda_{r,s}^{*} + a_{i,\mu_{r,s}} \mu_{r,s}^{*}$.

\begin{lemma}\label{lem:cup}
We have that 
\[
\langle \alpha_{i} \cup \alpha_{j}, [T_{r,s}] \rangle
= a_{i,\lambda_{r,s}} \delta_{j,s} - a_{j, \lambda_{r,s}} \delta_{i,s}.
\]
\end{lemma}

\proof
By a direct computation, we have:
\begin{eqnarray*}
\langle \alpha_i \cup \alpha_{j},  [T_{r,s}] \rangle 
&=&\langle (a_{i,\lambda_{r,s}} \lambda_{r,s}^{*} + a_{i,\mu_{r,s}} \mu_{r,s}^{*})  \cup (a_{j,\lambda_{r,s}} \lambda_{r,s}^{*} + a_{j,\mu_{r,s}} \mu_{r,s}^{*}) , [T_{r,s}] \rangle \\
&= &(a_{i,\lambda_{r,s}} a_{j,\mu_{r,s}} - a_{j,\lambda_{r,s}} a_{i,\mu_{r,s}}) 
\langle \lambda_{r,s}^{*} \cup \mu_{r,s}^{*}, [T_{r,s}] \rangle \\
&=&a_{i,\lambda_{r,s}} a_{j,\mu_{r,s}} - a_{j,\lambda_{r,s}} a_{i,\mu_{r,s}}.
\end{eqnarray*}
The Lemma follows from $a_{j,\mu_{r,s}} = \delta_{j,s}$ by the definition of the meridian. 
\endproof

Also, for the coefficient $a_{i, \lambda_{r,s}}$, we have 
\begin{eqnarray*}
a_{i, \lambda_{r,s}} = 
\left \{
\begin{array}{ll}
1 & (\mbox{$i \neq s$ and $H_r \cap H_s \in H_i$}), \\
0 & (\mbox{otherwise}).
\end{array}
\right .
\end{eqnarray*}

\textit{Proof of Theorem \ref{thm:main_cohom}}.
It suffices to show that the cup products satisfy 
\begin{equation*}
\alpha_i \cup \alpha_j=
\left \{ 
\begin{array}{ll}
\beta_{i,j} & (\mbox{if $(i,j) \in \mathbf{nbc}(\cA)$}) , \\
\beta_{k,j} - \beta_{k,i} & (\mbox{if there exists $k$ such that $(k,i), (k,j) \in \mathbf{nbc}(\cA)$}) , \\
0, & (\mbox{otherwise}),
\end{array}
\right .
\end{equation*}
for $1 \leq i < j \leq n$.

First, let $(i,j)$ be an nbc. Let $P=H_i \cap H_j$.
By Lemma \ref{lem:cup}, for $s \neq i,j$, 
$\langle \alpha_i \cup \alpha_j, [T_{r,s}] \rangle = 0$.
When $s=j$, 
$\langle \alpha_i \cup \alpha_j, [T_{r,s}] \rangle = a_{i,\lambda_{r,j}}$. If $a_{i,\lambda_{r,j}} \neq 0$, $H_{r} \cap H_j$ is contained in $H_i$. Thus, $P=H_r \cap H_j$. By the minimality of the index, we have $r=i$. 
When $s=i$, we have $ \langle \alpha_i \cup \alpha_j, [T_{r,s}] \rangle = -a_{j,\lambda_{r,i}}$. 
If $a_{j, \lambda_{r,i}} \neq 0$, then, $H_r \cap H_i$ is contained in $H_j$ and $P=H_r \cap H_i$. By the definition of an nbc, $r$ is the minimum index of this intersection. However, since $(i,j)$ is an nbc, this is a contradiction. 
Therefore, we have 
\begin{eqnarray*}
\langle \alpha_i \cup \alpha_j , [T_{r,s}] \rangle = 
\left \{
\begin{array}{ll}
1 & ((r,s) = (i,j)) ,\\
0 & (\mbox{otherwise}).
\end{array}
\right .
\end{eqnarray*}
In particular, $\alpha_{i} \cup \alpha_{j} = \beta_{i,j}$.

Next, suppose that there exists $k$ such that $(k,i)$ and $(k,j)$ are nbcs. Let $P=H_k \cap H_i \cap H_j$. 
Again, by Lemma \ref{lem:cup}, for $s \neq i,j$, 
$\langle \alpha_i \cup \alpha_j, [T_{r,s}] \rangle = 0$. 
When $s = j$, 
$\langle \alpha_i \cup \alpha_j, [T_{r,s}] \rangle = a_{i,\lambda_{r,j}}$. This is non-zero if $H_r \cap H_j$ is contained in $H_i$. In this case, $H_r \cap H_j = P = H_i \cap H_j$. By the minimality of the index, we have $r = k$. 
When $s=i$, by the same reason, we have $\langle \alpha_i \cup \alpha_j, [T_{r,s}] \rangle = -a_{j,\lambda_{r,i}}$ is non zero if and only if $k=r$. Therefore,  $\alpha_{i} \cup \alpha_{j} = \beta_{k,j} - \beta_{k,i}$. 

Let us consider the final case. In this case, $P =H_i \cap H_j$ is contained in $H_0$. By Lemma \ref{lem:cup}, for $s \neq i,j$, 
$\langle \alpha_i \cup \alpha_j, [T_{r,s}] \rangle = 0$. 
When $s=j$, 
$\langle \alpha_i \cup \alpha_j, [T_{r,s}] \rangle = a_{i,\lambda_{r,j}}$. If this does not vanish, $H_r \cap H_j$ is contained in $H_i$. 
Thus, $H_r \cap H_j = H_{i} \cap H_j \in H_0$. This contradicts that $(r,j)$ is nbc. When $s=i$, the pairing vanishes for the same reason. This completes the proof. 
\qed

\section{Homotopy types and minimality}\label{sec:min}

\subsection{Minimal CW complexes}

\begin{definition}
A finite CW complex $X$ is a \textit{minimal CW complex} if the $k$-th Betti number equals the number of $k$-cells in $X$ for each $k \geq 0$. 
A topological space $X$ is said to be \textit{minimal} if $X$ is homotopy equivalent to a minimal CW complex.
\end{definition}

\begin{remark}
Let $c_{k} (X)$ be the number of $k$-cells. 
Then, by the cellular chain complex, we always have that $b_k (X) \leq c_{k} (X)$. The minimality says that the equality holds for each $k$.
\end{remark}

Recall that the \textit{rank} $d(G)$ of a group $G$ is the minimum number of generators.
The following proposition implies that the minimality restricts the fundamental group.

\begin{proposition}\label{prop:min_fund}
Let $X$ be a connected minimal CW complex and $n=b_1 (X)$. Then, $d(\pi_1 (X)) = n$.
Moreover, if $b_1(X)=1$, then $\pi_1 (X) \cong \bZ$.
\end{proposition}

\proof
Recall that the $2$-skeleton of a CW complex $X$ gives a presentation of the fundamental group $\pi_1 (X)$ with $c_1 (X)$ generators and $c_2 (X)$ relations (\cite{hat}, Proposition 1.26). Thus, $d(\pi_1(X)) \leq c_1 (X)$.
Moreover, since $d(H_1(X)) =d(\pi_1(X)/ [\pi_1(X), \pi_1 (X)]) \geq n$, $d(\pi_1 (X)) \geq n$. Since $X$ is minimal, $c_1 (X)=n$. We have that $d(\pi_1 (X)) = n$.

When $n=1$, since $d(\pi_1 (X))=1$, $\pi_1 (X)$ is a cyclic group.
Since its abelianization has rank one, $\pi_1 (X) \cong \bZ$.
\endproof 

For example, a knot complement in the $3$-sphere is minimal if and only if the knot is a trivial knot, due to the unknotting theorem \cite[Theorem 4.B.1]{rol}.

\subsection{Symplectic line arrangements are minimal.}

In this subsection, we prove Theorem \ref{thm:main_symp}, which states that the complement of a symplectic line arrangement is minimal.
To prove this, we use Libgober's theorem describing the homotopy type of the complement of algebraic and $J$-holomorphic curves. 
\begin{proposition}\label{prop:lib-hom}
(cf. \cite{lib-hom}). 
Let $C \cup L$ be a $J$-holomorphic symplectic curve in $\bC P^2$, where $L$ is a $J$-holomorphic line. Then, the complement $\bC P^2 \setminus (C \cup L)$ is homotopy equivalent to the $2$-dimensional CW complex associated to the presentation of the fundamental group given by the braid monodromy.
\end{proposition}

Libgober proves this statement for algebraic curves. His argument, however, uses only the local triviality of the fibration induced by the pencil away from its singular fibers and the standard two-dimensional CW model for a link complement associated with a Wirtinger/Artin presentation. Both ingredients remain valid for $J$-holomorphic symplectic curves. Hence the same proof applies without change.

\vspace{3mm}

\textit{Proof of Theorem \ref{thm:main_symp}}.
We apply the braid monodromy technique for symplectic curves in $\bC P^2$ (see \cite[Section 4]{kha-kul} or \cite[Section 3.1]{awa-gol} for details).
Let $\omega_{FS}$ be the Fubini-Study symplectic form and $J$ be a compatible almost complex structure such that $V(\cA)$ is $J$-holomorphic, and each line $H \in \cA$ is a $J$-holomorphic line (see \cite{sta}).
Let us fix the line $H_{0} \in \cA$ and take $p \in H_{0} \setminus \Sing_{0} (\cA) $.
Also, we take a generic $J$-holomorphic line $L$ such that $L \notin \cA$.
It is known that for two distinct points $p,q \in \bC P^2$, there exists a unique $J$-holomorphic line $L_{p,q}$ passing through $p$ and $q$. 
Moreover, two distinct $J$-holomorphic lines intersect transversally at a single point.
Define a projection by $\pi: \bC P^2 \setminus \{p\} \rightarrow L \cong \bC P^1; \pi(q) = L_{p,q} \cap L$. 
As in the  usual braid-pencil construction for complex curves, it is known that the restriction $\pi: U(\cA) \setminus \pi^{-1} (X) \to \bC P^1 \setminus X$ is a locally trivial fibration whose fiber is a complex plane with $n$ points removed, 
where $X$ is the set of points $q$ such that $\pi^{-1}(q)$ coincides with $H_{0}$, $\pi^{-1}(q)$ contains a singularity of $\bigcup_{i=1}^{n} H_{i}$, or $\pi^{-1} (q)$ is tangent to $\bigcup_{i=1}^{n} H_{i}$.
Note that no tangencies occur since two distinct $J$-holomorphic lines must intersect transversally. 

Now, let us recall the presentation of the fundamental group of the complement.
Take a point $p_{0} \in \bC P^1 \setminus X$. Since the fiber $\pi^{-1}(p_0)$ is a complex plane with $n$ punctures, the fundamental group of the fiber is free of rank $n$. 
A presentation of the fundamental group of the complement is obtained from this free group by adding the relations corresponding to each degenerate fiber and braid monodromy around it. By the symplectic version of Libgober's argument summarized in Proposition \ref{prop:lib-hom}, the associated CW complex is homotopy equivalent to the complement.

In our case, the resulting connected $2$-dimensional CW complex consists of $n$ $1$-cells. Thus, the number of $2$-cells $c_{2}$ satisfies
\[
c_2 = \chi (U(\cA)) +n-1 = (|\mathbf{nbc} (\cA)|-n+1 ) +n-1 = |\mathbf{nbc}(\cA)| = b_2 (U(\cA)).
\]
This proves $U(\cA)$ is minimal.
\endproof

\subsection{Non-minimal arrangements and infinite families.}
In this subsection, we construct a topological line arrangement whose complement is not minimal. We will show the non-minimality by studying the fundamental group and Proposition \ref{prop:min_fund}.
Moreover, this idea enables us to produce an infinite family of arrangements with the same combinatorial type and pairwise non-homotopy equivalent complements.
Before proceeding to the construction of the examples, we prepare a key lemma.

\begin{lemma}\label{lem:2-knot}
For every positive integer $N$, there exists a $2$-knot $F_N \subset S^4$ such that 
$d(\pi_1 (S^4 \setminus F_N)) \geq N+1$. 
\end{lemma}

\proof 
Let $K$ be a connected sum of $N$ non-trivial knots. By \cite[Theorem 2]{wei}, its knot group has rank at least $N+1$. The spun knot $F_{N} :=\nu (K)$ has the same knot group \cite{artin} (see also \cite[Lemma 6.2.2]{kam}). 
\endproof

As one of the examples of such $F_N$, we take the spun $\nu(\sharp_{N} 3_1)$ of connected sums of $N$ trefoils.
It is known that the fundamental group of the complement of $F_{N}$ has the following Wirtinger presentation:
\[
\pi_1 (S^4 \setminus F_N) 
= \langle 
a_i, b_i, c_i
\mid 
c_i = a_i b_i a_i^{-1},
b_i = c_i a_i c_i^{-1},
a_i = b_i c_i b_i^{-1},
c_i = a_{i+1}
\rangle,
\]
where $i=1,\dots,N$.

Now, we are ready to construct topological line arrangements with non-minimal complement.

\vspace{3mm}

\textit{Proof of Theorem \ref{thm:main_nonmin}}. 
Let $\cA = \{H_0, \dots, H_n\}$ be a topological line arrangement.
Let $G = \pi_1 (U(\cA))$ and give a presentation 
\[
G = \langle x_1, \dots, x_m \mid r_1, \dots, r_s\rangle.
\]
We can assume that the meridian around some line is contained in the generators (if any meridian does not appear, then add a meridian to the generators). 
Let $x_{m} = m_{k}$ the meridian around $H_{k}$ ($k \in \{1,\dots, n\}$). 
The abelianization $G^{ab} \cong H_1 (U(\cA); \bZ) \cong \bZ^{n}$ is freely generated by the meridians $m_1, \dots, m_n$, and thus expressed as 
\[
G^{ab} = \langle m_1, \dots, m_{k}=x_{m} , \dots, m_{n} \mid [m_i, m_j]\rangle.
\]
Now, by taking the connected sum of $H_{k}$ and the spun $F_n$ of connected sums of $n$ trefoils, we obtain a new topological line arrangement $\cA' = \{H_0,\dots, H_k \sharp F_{n},\dots, H_n\}$. Let $G' = \pi_1 (U(\cA'))$. By the Seifert-van Kampen theorem, we have the following presentation for $G'$.
\[
G' = \langle x_1, \dots, x_m, a_1, \dots, c_n 
\mid r_1, \dots, r_s, x_m = a_1, \mbox{(Relations of $\pi_1 (S^4 \setminus F_n)$)} \rangle.
\]
By using the abelianization map, this group admits a surjective homomorphism to the group $G''$ defined by 
\[
G'' = \langle m_1, \dots ,m_n,  a_1, \dots, c_n 
\mid [m_i, m_j], m_k = a_1, \mbox{(Relations of $\pi_1 (S^4 \setminus F_n)$)} \rangle.
\]
Finally, let us define the homomorphism $\varphi: G'' \to \pi_1 (S^4 \setminus F_n)$ by $\varphi (m_{i}) = 1$ ($i \neq k$), $\varphi (m_k) = a_{1}$, and sending $a_i,b_i,c_i$ to the same elements. Then $\varphi$ is surjective.
Since $d(\pi_1 (S^4 \setminus F_n)) \geq n+1$, we have $d(G') \geq n+1$.
Thus by Proposition \ref{prop:min_fund}, $U(\cA')$ is not minimal.
\endproof

We can use this construction to produce an infinite family of arrangements whose combinatorial types are isomorphic and the complements are pairwise non-homotopy equivalent.

\vspace{3mm}

\textit{Proof of Theorem \ref{thm:main_infty}}.
Let $\cA = \{H_0,\dots, H_n\}$ be a topological line arrangement. As in the proof of Theorem \ref{thm:main_nonmin}, we can construct a topological line arrangement $\cA'_{(N)}$ such that $d (\pi_1 (U(\cA'_{(N)}))) \geq N+1$ by taking a connected sum of one component with $F_{N}$ for any $N \in \bZ_{>0}$. 
Thus, it follows that 
\[
\sup\{d(\pi_1 (U(\cA'))) \mid L(\cA) \cong L(\cA') \} = \infty,
\]
for any $\cA$. 
%Thus, while keeping the combinatorial type of the arrangement fixed, the rank of the fundamental group of its complement can be made arbitrarily large. 
Since the ranks are unbounded, we may inductively choose a sequence of such realizations whose fundamental groups have strictly increasing ranks. 
Consequently, one can construct infinitely many arrangements with the same combinatorial type whose complements have pairwise distinct homotopy types.
\endproof

This idea enables us to show that any $2$-knot complement appears as the complement of a topological line arrangement.

\begin{theorem}\label{thm:main_knot}
Let $F$ be any $2$-knot. Then, there exists a topological line arrangement $\cA_{F}$ such that $U(\cA_{F})$ is diffeomorphic to $S^4 \setminus F$.
\end{theorem}

\proof
Let $\cA = \{H_0, H_1\}$ be a complex algebraic line arrangement of two lines. 
Then, it is easy to see that the complement $U(\cA)$ is diffeomorphic to $\bC^{*} \times \bC \cong S^1 \times \bR^3$, which is also diffeomorphic to a trivial $2$-knot complement.
Let $H'_1 = H_1 \sharp F$ be the connected sum of $H_1$ and $F$ and we set $\cA_{F} =\{H_0, H_1'\}$. Then, $U(\cA_{F})$ is diffeomorphic to $S^4 \setminus F$.
\endproof 

\subsection{Relationship with $2$-pseudo arrangements}
The notion of $2$-pseudoarrangements, introduced by Bj\"orner and Ziegler \cite{bjo-zie}, is also a topological generalization of complex hyperplane arrangements. 
A $2$-pseudoarrangement is an arrangement of $(2d-3)$-spheres in $S^{2d-1}$ with appropriate conditions, see \cite[Definition 8.1 and 8.3]{bjo-zie} for precise definition. 
Topological line arrangements, which are considered in our paper, are related to $2$-pseudoarrangements in $S^5$.
In fact, if each $3$-sphere in a $2$-pseudoarrangement in $S^5$ is invariant under the natural $S^1$-action, 
then it induces an arrangement of $2$-spheres in $\bC P^2$ via the Hopf fibration.
By condition (i) in \cite[Definition 8.3]{bjo-zie}, the intersection of any two $3$-spheres is a circle. Thus, the induced arrangement of $2$-spheres in $\bC P^2$ satisfies the condition that any two $2$-spheres intersect in exactly one point.
Therefore, by giving an appropriate orientation, the resulting arrangement is a topological line arrangement.

Conversely, from a topological line arrangement in $\bC P^2$, we obtain an arrangement of $3$-spheres in $S^5$ by pulling back the Hopf fibration. 
Here, we call such an arrangement of $3$-spheres in $S^5$ the \textit{lifted topological arrangement}. 

It is natural to ask whether any lifted topological arrangement is a $2$-pseudo arrangement. However, the class of lifted topological arrangements is strictly larger than $2$-pseudoarrangements. 

\begin{theorem}\label{thm:lifted}
There exists a lifted topological arrangement which is not a $2$-pseudoarrangement.
\end{theorem}

\proof
Let $G$ be a non-abelian $2$-knot group (for example, the group of the spun trefoil).
By Theorem \ref{thm:main_knot}, there exists a topological line arrangement $\cA' = \{H_0, H'_1\}$ such that $\pi_1 (U(\cA')) = G$.
Let $\overline{\cA'} = \{\overline{H_0}, \overline{H'_1} \}$ the corresponding lifted arrangement with the complement $U(\overline{\cA'}) = S^5 \setminus (\overline{H_0} \cup \overline{H'_1})$. 
Since $H^2 (U(\cA')) =0$, the restriction of the Hopf fibration to $U(\cA')$ is trivial. Thus, it follows that $U(\overline{\cA'}) \cong U(\cA') \times S^1$.
In particular, we have $\pi_1 (U(\overline{\cA'})) \cong G \times \bZ$.
Since $G$ is non-abelian, so is $\pi_1 (U(\overline{\cA'}))$.

Suppose that $\overline{\cA'}$ is a $2$-pseudoarrangement in $S^5$. 
By the condition (ii) in \cite[Definition 8.3]{bjo-zie}, there exists an arrangement $\cB'=\{T_0, T_1, U_0, U_1\}$ of $4$-spheres in $S^5$ such that $H_0= T_0 \cap U_0$ and $H'_1 = T_1 \cap U_1$ as the real frame. 
We regard the real frame as a signed arrangement by choosing a positive side for each pseudosphere. 
Let $(z_0, z_1, z_2)$ be complex coordinates on $S^5$ and let $z_j = x_{j} + \sqrt{-1}y_{j}$. Then, by defining $C_{j} = \{x_j = 0\}, D_{j} = \{y_j = 0\}$, we have an arrangement of pseudospheres $\cC = \{C_0, C_1, D_0, D_1\}$, whose combinatorial structure is isomorphic to $\cB'$. Thus, by \cite[Theorem 5.1.6]{BLVSWZ}, the pseudosphere arrangements $\cB'$ and $\cC$ are topologically equivalent; in particular, there exists a homeomorphism $h: S^5 \to S^5$ such that $h(T_j) = C_j$ and $h(U_j) =D_j$ ($j=0,1$). Thus, the two complements of $2$-pseudoarrangements $U(\overline{\cA'})$ and $S^5 \setminus \bigcup_{j=0,1} (C_j \cap D_j)$ are homeomorphic.
It is easy to see that the fundamental group of the complement of two complex coordinate $3$-spheres is isomorphic to $\bZ \times \bZ$. 
However, this contradicts that $\pi_1 (U(\overline{\cA'}))$ is non-abelian.
\endproof

\section{Appendix: Intersection ring and intersection dual}\label{sec:app}

%\subsection{Intersection ring and intersection dual}\label{subsec:ring}
Here, we recall the basic facts on the cohomology ring and its dual object, the homology intersection ring, see \cite[Chapter VI]{bre} for details.
Let $M$ be a compact oriented $n$-manifold with boundary. Recall the Poincar\'e-Lefschetz duality:
\[
D: H^{k} (M) \cong H_{n-k} (M, \bd M); \alpha \mapsto D(\alpha) = \alpha \cap [M, \bd M].
\]
The homology intersection product is defined as 
\[
\cdot: H_{k_1} (M, \bd M) \times H_{k_2} (M, \bd M) \to H_{k_1+k_2 - n} (M, \bd M);
(a,b) \mapsto D(D^{-1}(a) \cup D^{-1} (b)).
\] 
By the Poincar\'e-Lefschetz duality, $H_{n-*}(M, \bd M)$ becomes a graded algebra which is isomorphic to the cohomology ring. 
The intersection product is computed geometrically. Suppose that the homology class $a_{i} \in H_{k_1} (M, \bd M)$ is represented by oriented submanifolds $X_i$ of $M$ intersecting transversally ($i=1,2$). 
Then, the intersection product satisfies $a_1 \cdot a_2 = [X_1 \cap X_2]$. Note that $a_2 \cdot a_1 = [X_2 \cap X_1] = [(-1)^{(n-k_1)(n-k_2)} X_1 \cap X_2]= (-1)^{(n-k_1)(n-k_2)} a_1 \cdot a_2$. 

Now, assume that $H_{k}(M)$ is a free module for each $k \geq 0$. Then, by the Poincar\'e-Lefschetz duality and universal coefficient theorem, we have
\[
(H_{k} (M))^{*} \cong H_{n-k} (M, \bd M).
\]
Thus, the relative homology group $H_{n-k} (M, \bd M)$ is the dual of $H_{k} (M)$.
Let $\{x_1, \dots, x_l\}$ be a basis of $H_k (M)$ with the dual basis $\{x_1^{\vee}, \dots , x_{l}^{\vee}\}$. 
Using the duality above, this dual basis is taken such that $x_{i}^{\vee} \cdot x_{j} = \delta_{i,j} $ via the intersection product
\[
\cdot :  H_{n-k} (M, \bd M) \times H_{k} (M) \to \bZ; (x_{i}^{\vee}, x_j ) \mapsto \delta_{i,j}.
\] 
We call this dual cycle the \textit{intersection dual} of $x_{i}$.

\end{document}